\documentclass[11pt]{article}
\usepackage{amsmath,mathrsfs,amssymb,amsthm,amscd}
\usepackage[]{fontenc}
\usepackage{xy}
\usepackage{enumerate}
\xyoption{all}

\numberwithin{equation}{section}

\theoremstyle{plain}
\newtheorem{theorem}{Theorem}[section] % これを基準（ベース）にする

\newtheorem{proposition}[theorem]{Proposition}
\newtheorem{lemma}[theorem]{Lemma}

\newtheorem{definition}[theorem]{Definition}
\newtheorem{notation}[theorem]{Notation}

\renewcommand{\thefootnote}{}

\title{Improved bounds on the number of holomorphic maps
between compact Riemann surfaces}

\author{
Masaharu Tanabe\thanks{Research supported by JSPS KAKENHI Grant Number 21K03287.}
}

\date{}

\begin{document}

\maketitle

\begingroup
\renewcommand{\thefootnote}{}
\endgroup

\begin{abstract}
 \small{We give new upper bounds, depending only on the genus, for the number of nonconstant holomorphic maps between compact Riemann surfaces.
Our estimates improve previously known bounds.
The proof is based on the study of pullbacks of holomorphic differentials, together with techniques from the geometry of numbers and the theory of Jacobian varieties.}
\end{abstract}

\section{INTRODUCTION} 

 Let $X$ be a compact Riemann surface of genus $g(>1)$. De Franchis \cite{F} stated the following: 

\vspace{0.2cm}
 \textbf{Theorem of de Franchis.} (a) \textit{For a fixed compact Riemann surface $Y$ of genus $>1$, the number of nonconstant holomorphic maps $X\rightarrow Y$ is finite.} \\
(b) \textit{There are only finitely many compact Riemann surfaces $Y_{i}$ of genus $> 1$ which admit a nonconstant holomorphic map from $X$.} 

\vspace{0.2cm}
 The second statement (b) is often attributed to Severi. 
 Once the finiteness of maps is established, finding an upper bound that depends only on a topological invariant, such as genus \(g\), becomes a problem of significant interest.
 Related to the statement (a), Martens \cite{Mr} showed that if one fixes a Riemann surface $Y$, then the number of all nonconstant holomorphic maps from $X$ to $Y$ is less than $(cg)^{2g^{2}}$ for some constant $c>1$ independent of $g$. The author \cite{T1} showed that the bound is smaller than $(cg)^{2g}$ for some constant $c$. 

 Now, we consider a bound for holomorphic maps when $Y$ is not fixed, that is, we estimate the number of all nonconstant holomorphic maps from $X$ to other Riemann surfaces. Let $f_{i}:X\rightarrow Y_{i}$ be nonconstant holomorphic maps for $i=1,2$. 
We say that $f_{1}$ and $f_{2}$ are isomorphic if and only if there is a conformal map $h:Y_{1}\rightarrow Y_{2}$ such that $h\circ f_{1}=f_{2}$. Let $\mathcal{I}_{\gamma}(X)$ denote the set of all isomorphic classes of nonconstant holomorphic maps into compact Riemann surfaces of genus $\gamma>1,$ and denote $\mathcal{I}(X)=\bigcup_{g>\gamma>1}\mathcal{I}_{\gamma}(X)$. By the theorem of de Franchis, we see that $\#\mathcal{I}(X)$ is finite. 

In 1983 Howard and Sommese \cite{H-S} first showed that there is a bound on
$\#\mathcal I(X)$ depending only on $g$.
In 1986 Kani \cite{K} gave a bound for the number of morphisms from a smooth
projective curve defined over a field $\mathbb K$, depending only on the genus
of the curve. In particular, when $\mathbb K=\mathbb C$, he proved that
\begin{equation*}
\#\mathcal I(X)
<
(g-1)2^{2g^{2}-2}(2^{2g^{2}-1}-1).
\end{equation*}

In 1990 Alzati and Pirola \cite{A-P} obtained the improved estimate
\begin{equation*}
\#\mathcal I(X)
<
\exp\{(4/3)(g^{2}-1)\log 3
+[\log_{2}g]\log(84g)
+\log(12\sqrt2)\}.
\end{equation*}

The author \cite{T2} later showed that
\begin{equation*}
\#\mathcal I(X)
<
(2g)^{4g}\times2^{2g-3}
\times(2g-1)^{g-1}(2g-3)(g-2)(g-1).
\end{equation*}

On the other hand, bounds for the number of holomorphic maps to a fixed
target have also been studied.
Let $\mathcal N(X,Y)$ denote the number of nonconstant holomorphic maps
from $X$ to a compact Riemann surface $Y$ of genus $\gamma >1$.
In 2007, Naranjo and Pirola \cite{N-P} proved that
\begin{equation*}\label{NP}
\mathcal N(X,Y)
\le
8(g-1)\rho
\left(
\binom{2g}{1}(2\rho)^{2g-1}
+
\binom{2g}{3}(2\rho)^{2g-3}
+\cdots
\right),
\end{equation*}
where
$
\rho=(g-1)/(\gamma -1).
$

In 2011, Chamizo and Fuertes \cite{Ch} obtained new bounds for the numbers
of holomorphic maps. 
Setting
\[
N_d=N_d(X,Y)
=\#\{f:\text{holomorphic map }X\to Y
\text{ with }\deg(f)\le d\}
\]
and
\[
I_d=I_d(X)
=\sum_{Y}N_d(X,Y),
\]

they gave
\[
N_d\le 8(g-1)(2d)^{2g},
\]
and
\begin{equation}\label{ineq}
I_d\le
\binom{2g-2}{d}
(2d)^{2g+1}(2g-1)^d.
\end{equation}
They also provided a detailed comparison of these upper bounds.

In this paper, we  improve the bound.

\begin{theorem}\label{thm}
Let $X$ be a compact Riemann surface of genus $g>1$. Then the number $\#\mathcal{I}(X)$ satisfies
\begin{equation}\label{eq1}
\#\mathcal{I}(X)\le
\left\{ \left(2g-1\right)^{2g}
+\left(2g-1\right)^{2g-1}+\cdots +\left(2g-1\right)^{4}\right\}  \binom{6g-6}{2g-2}.
\end{equation}
Without using the combinatorial notation,
\begin{equation}\label{eq2}
\#\mathcal I(X)
\le
\left\{
(2g-1)^{2g}
+\cdots+
(2g-1)^4
\right\}
\left(\frac{27}{4}\right)^{2g-2}.
\end{equation}
\end{theorem}

From the Riemann-Hurwitz formula, the greatest possible degree is $g-1$.
Substituting $d=g-1$ into (\ref{ineq}), 
we see that the ratio of (\ref{eq1}) to (\ref{ineq}) tends to \(0\) as \(g\) tends to infinity.

We check the asymptotic order of (\ref{eq2}) as 
$g\to\infty$.
We have
$$
\sum_{k=4}^{2g}(2g-1)^k
=
(2g-1)^{2g}
\left(
1+\frac{1}{2g-1}+\cdots+\frac{1}{(2g-1)^{2g-4}}
\right).
$$
Since
$$\sum_{k=0}^{\infty}(2g-1)^{-k}
=
\frac{1}{1-(2g-1)^{-1}}
=
1+O(g^{-1}),$$
we see that
$$
\left\{
(2g-1)^{2g}
+(2g-1)^{2g-1}
+\cdots
+(2g-1)^4
\right\}
\left(\frac{27}{4}\right)^{2g-2}
$$
has the asymptotic order
$$
\asymp
(2g)^{2g}
\left(\frac{27}{4}\right)^{2g-2}=\left(\frac{27}{2}g\right)^{2g}\left(\frac{4}{27}\right)^{2}.
$$

 The proof of Theorem \ref{thm} is based on estimating the possible number of pull backs of harmonic differentials in $H^{1}(Y_{i},\mathbb{Z})$. We may identify $H^{1}(Y_{i},\mathbb{Z})$ with the lattice of the dual Jacobian variety $\widehat{J(Y_{i})}$. Thus, we will use theory of homomorphisms of Jacobians and theory of lattices from the geometry of numbers. 

 Let 
\begin{equation*}
M(g) = \max_X \{\#\mathcal{I}(X)\}, 
\end{equation*}
where the maximum is taken over all Riemann surfaces $X$ of genus $g$.
Kani [K]  constructed a sequence $F_{1},F_{2},...,F_{n},...$ of function fields over $\mathbb{K}$ of genus $g_{F_{1}}<g_{F_{2}}<...<g_{F_{n}}<...,$ such that the number of separable subfields of $F_{n}/\mathbb{K}$ is greater than $\exp (c(\log(g_{F_{n}}))^{2})$ for some constant $c>0$ (independent of $n$).  This implies that $M(g)$ cannot be bounded by any polynomial in $g$.

\section{NOTATION AND LEMMATA} 

Throughout this paper, 
all Riemann surfaces are compact and of genera greater than 1. First, we recall some notions from complex tori.  Let $V$ be a complex vector space and $\Gamma$ a lattice in $V$. The quotient $T=V/\Gamma$ is called a complex torus.  Denote by $ \widehat{T}=V^{*}/ \widehat{\Gamma}$ the dual where $V^{*}$ is the space of $\mathbb{C}$-antilinear functionals on $V$ and $ \widehat{\Gamma}=\{l\in V^{*}:\text{Im } l(\Gamma)\subseteq\mathbb{Z}\}$ is the dual lattice of $\Gamma$. Let $\frak{f}$ be a homomorphism between two complex tori $T=V/\Gamma$ and $T^{\prime}=V^{\prime}/\Gamma^{\prime}$.  Then, there is a unique $\mathbb{C}$-linear map $F:V\rightarrow V^{\prime}$ with $F(\Gamma)\subseteq\Gamma^{\prime}$ inducing $\frak{f}$.  We call $F$ the {\it analytic representation} of $\frak{f}$, and the restriction $F|_{\Gamma}$ the {\it rational representation} of $\frak{f}$.  For the analytic representation $F:V\rightarrow V^{\prime}$ of a homomorphism $\frak{f}:T\rightarrow T^{\prime}$, the dual map $^{t}F:V^{\prime *}\to V^{*}$ associating to an antilinear functional $l\in V^{\prime^{*}}$ the antilinear functional $l\circ F\in V^{*}$ induces a homomorphism $^{t}\frak{f}: \widehat{T}^{\prime}\rightarrow \widehat{T}$ since $^{t}F( \widehat{\Gamma^{\prime}})\subseteq \widehat{\Gamma}.$ We call $^{t}\frak{f}$ the dual map of $\frak{f}$. 

 Let $X$ and $Y$ be compact Riemann surfaces of genera $g$ and $\gamma$, respectively.  Denote by $\mathcal{H}$ the space of holomorphic differentials on $X$. Set $\Omega=\text{Hom}(\mathcal{H},\mathbb{C})$.  The Jacobian variety $J(X)=\Omega/H_{1}(X,\mathbb{Z})$ is a complex torus of dimension $g$ and considering $\mathcal{H}$ of $\mathbb{C}$-antilinear forms on $\Omega$, we denote by $ \widehat{J(X)}=\overline{\mathcal{H}}/H^{1}(X,\mathbb{Z})$ the dual. 

 In order to describe $J(X)$ in terms of period matrices, choose basis $\lambda_{1},...,\lambda_{2g}$ of $H_{1}(X,\mathbb{Z})$ and $\omega_{1},...,\omega_{g}$ of $\mathcal{H}$. Let $l_{1},...,l_{g}$ be the basis of $\Omega$ dual to $\omega_{1},...,\omega_{g}$, i.e., $l_{i}(\omega_{j})=\delta_{ij}$.  Considering
 $\lambda_{j}$ as a linear form on $\mathcal{H}$, we have $\lambda_{j}=\sum_{k=1}^{g}(\int_{\lambda_{j}}\omega_{k})l_{k}$ for $j=1,...,2g$. Hence 
\begin{equation*}
\Pi_{X}=\begin{pmatrix}\int_{\lambda_{1}}\omega_{1} & ... & \int_{\lambda_{2g}}\omega_{1} \\ \vdots & & \vdots \\ \int_{\lambda_{1}}\omega_{g} & ... & \int_{\lambda_{2g}}\omega_{g} \end{pmatrix} \text{\quad }
\end{equation*}
is a period matrix for $J(X)$ with respect to these bases.  Similarly, we define $J(Y)$, $ \widehat{J}(Y)$ and $\Pi_{Y}$ for $Y$. Let $\frak{f}:J(X)\rightarrow J(Y)$ be a homomorphism.  In terms of matrices, $\frak{f}$ can be expressed as 
\begin{equation*}
A\Pi_{X}=\Pi_{Y}M \text{\quad }
\end{equation*}
with $A\in M(\gamma,g;\mathbb{C})$ and $M\in M(2\gamma,2g;\mathbb{Z})$.  Conversely, if there are matrices $A\in M(\gamma,g;\mathbb{C})$ and $M\in M(2\gamma,2g;\mathbb{Z})$ such that $A\Pi_{X}=\Pi_{Y}M$ then these matrices are matrix representations of some homomorphism $J(X)\rightarrow J(Y)$.  We also call $A$ the analytic representation of $\frak{f}$, and $M$ the rational representation of $\frak{f}$. 

 There is a canonical principal polarization on $J(X)$. Fix a homology basis $\lambda_{1},...,\lambda_{2g}$ of $H_{1}(X,\mathbb{Z})$ with an intersection matrix 
\begin{equation*}
J=\begin{pmatrix}0&I\\ -I&0\end{pmatrix} \text{\quad }
\end{equation*}
where each entry is $g\times g$ sized.  Considered as a $\mathbb{R}$-vector space, $\lambda_{1},...,\lambda_{2g}$ is a basis for $\Omega$. Denote by $E$ the alternating form on $\Omega$ with the matrix $J$ with respect to the basis $\lambda_{1},...,\lambda_{2g}$ for $\Omega$ and define a hermitian form $H:\Omega\times\Omega\rightarrow\mathbb{C}$ by 
\begin{equation*}
H(u,v)=E(iu,v)+iE(u,v). \text{\quad }
\end{equation*}
We denote by $\langle\zeta,v\rangle$, $\zeta\in\overline{\mathcal{H}},v\in\Omega$ the value of $\zeta$ at $v$. The $\mathbb{C}$-linear map $\phi_{E}:\Omega\rightarrow\overline{\mathcal{H}}$ such that $H(u,v)=\langle\phi_{E}(u),v\rangle$ induces a homomorphism, also denoted by $\phi_{E}$, 
\begin{equation*}
\phi_{E}:J(X)\rightarrow \widehat{J(X)}. \text{\quad }
\end{equation*}
The matrix $-J=J^{-1}=\begin{pmatrix}0&-I\\ I&0\\ \end{pmatrix}$ is the rational representation of $\phi_{E}$ with respect to the homology basis $\lambda_{1},...,\lambda_{2g}$ for $H_{1}(X,\mathbb{Z})$ above and the dual basis $\alpha_{1},...,\alpha_{2g}$ for $H^{1}(X,\mathbb{Z})$, i.e., $\alpha_{j}(\lambda_{k})=\delta_{jk}.$ Also, we can transport $E$ to $\overline{\mathcal{H}}$ by defining 
\begin{equation*}
E^{\prime}(\phi_{E}(u),\phi_{E}(v))=E(u,v). \text{\quad }
\end{equation*}
We obtain a bound for $\#\mathcal{I}(X)$ by counting the possible numbers of pull backs of holomorphic differentials.  Thus, we mainly deal with the dual Jacobian varieties and the endomorphisms of them.  By an {\it underlying real structure} for a $g$-dimensional complex torus $T=V/\Gamma$, we mean the real torus $\mathbb{R}^{2g}/\mathbb{Z}^{2g}$ together with a map $\mathbb{R}^{2g}/\mathbb{Z}^{2g}\rightarrow T$ induced by a linear map $\mathbb{R}^{2g}\ni x\mapsto\Pi x\in\mathbb{C}^{g}$ where $\Pi$ is a period matrix.  For Jacobian varieties, it is known that the real part $E(iu,v)$ of the hermitian $H(u,v)$ is 
symmetric positive definite. 

 Here, we define an inner product. 

\begin{definition}
On $\mathbb{R}^{2g}$ which is over the underlying real structure for $J(X)$, we define an inner product by 
\begin{equation*}
(x,y)_{X}=E^{\prime}(i \widehat{\Pi}x, \widehat{\Pi}y),
\end{equation*}
where $ \widehat{\Pi}$ is a period matrix for $ \widehat{J(X)}$. We define a norm $|| \cdot ||$ by 
\begin{equation*}
||x||=\sqrt{(x,x)_{X}}. \text{\quad }
\end{equation*}
\end{definition}

For an arbitrary $\mathfrak{g}\in End(J(X))$, put $\mathfrak{g^{\prime}}=\phi_{E}^{-1}\circ^{t}\mathfrak{g}\circ\phi_{E}$ and ${}^t \mathfrak{g}^{\prime}=\phi_{E}\circ\mathfrak{g}\circ{\phi_{E}}^{-1}.$ (We observe $(^{t}\mathfrak{g})^{\prime}={^{t}(\mathfrak{g^{\prime}})}$ holds.) This map $': End(J(X))\rightarrow End(J(X))$ (or $End( \widehat{J(X)})\rightarrow End( \widehat{J(X)}))$ is so-called the Rosati involution.  For a nonconstant holomorphic map $f:X\rightarrow Y$ denote by $\frak{f}$ the induced homomorphism between the Jacobians.  Put $\frak{f}^{\prime}=\phi_{E}^{-1}\circ^{t}\frak{f}\circ\phi_{E^{\prime}}$, where $\phi_{E^{\prime}}:J(Y)\rightarrow \widehat{J(Y)}$ is defined similar  to $\phi_{E}:J(X)\rightarrow \widehat{J(X)}$. If we denote by $J^{\prime}=\begin{pmatrix}0&I\\ -I&0\end{pmatrix}$ where each entry is $\gamma\times\gamma$, then $-J^{\prime}=J^{\prime -1}$ is the rational representation of $\phi_{E^{\prime}}:J(Y)\rightarrow \widehat{J(Y)}.$ We call $\mathfrak{F}=\frak{f}^{\prime}\circ \frak{f}(\in End(J(X)))$ the {\it endomorphism associated with} $f$. 

\vspace{0.2cm}
\begin{lemma}\label{lem1} For an arbitrary $\mathfrak{g}\in End(J(X))$, denoting by $G$ and $^{t}G^{\prime}$ the rational representations of the duals $^{t}\mathfrak{g}$ and $^{t}\mathfrak{g^{\prime}}$, respectively, we have
\begin{equation}
(^{t}Gx,^{t}Gy)_{X}=(^{t}G^{\prime}{^{t}Gx},y)_{X}=(x,^{t}G^{\prime}{^{t}Gy})_{X}, 
\end{equation}
for all $x, y\in\mathbb{R}^{2g}.$ In particular, when $\mathfrak{F}$ is the endomorphism associated with some holomorphic map $f:X\rightarrow Y$ of degree $d$, denoting by $\mathcal{F}$ the rational representation of $\mathfrak{F}$, we have
\begin{equation}
(^{t}\mathcal{F}x,^{t}\mathcal{F}y)_{X}=d(^{t}\mathcal{F}x,y)_{X}. 
\end{equation}
\end{lemma}

\vspace{0.2cm}
 \begin{lemma} \label{lem2} 
 Let $f:X\rightarrow Y$ be a nonconstant holomorphic map of degree $d$, and let $\mathfrak{F}$ be the endomorphism associated with $f$. 
 Then, denoting by $\mathcal{F}$ the rational representation of $\mathfrak{F}$, we have
\begin{equation}
||^{t}\mathcal{F}x||\le d||x||, 
\end{equation}
for an arbitrary vector $x\in\mathbb{R}^{2g}$.
\end{lemma}

Proofs for the above lemmata are in \cite{T2}.

The following lemma is the key geometric ingredient.
It provides an estimate for the number of holomorphic maps having a prescribed pullback of a holomorphic differential.
 \begin{lemma}\label{lem} Let $f_{1}:X\rightarrow Y_{1}$ be a holomorphic map of degree $d$, and let ${\mathfrak f}_{1}:J(X)\rightarrow J(Y_{1})$ be the homomorphism induced by $f_{1}$.  Take an arbitrary $u\in{^{t}{\mathfrak f}_{1}( \widehat{J(Y_{1})})}$. Then, the number of isomorphic classes of holomorphic maps $f_{i}:X\rightarrow Y_{i}$ of degree $d$ such that the dual map $^{t}{\mathfrak f}_{i}: \widehat{J(Y_{i})}\rightarrow  \widehat{J(X)}$ of the induced homomorphism ${\mathfrak f}_{i}$ satisfies $u\in{^{t}{\mathfrak f}_{i}( \widehat{J(Y_{i})})}$ is at most $\binom{2g-2}{d}\times\binom{4g-4}{d}$.
\end{lemma}
In \cite{T2}, the corresponding bound is
$\binom{2g-2}{d}\times (2g-1)^{d},
$ which is improved here to
$\binom{2g-2}{d}\times\binom{4g-4}{d}.$

\begin{proof}
The assumption implies that there exist holomorphic differentials
$\phi_{1}
$ on \(Y_{1}\) and
$\phi_{i}
$ on \(Y_{i}\) such that 
$f_{1}^{*}\phi_{1}=f_{i}^{*}\phi_{i}.
$

Let \(p_{01}\) be a zero of \(\phi_{1}\). Then the number of possible
$f_{1}^{-1}(p_{01})
$ (counting multiplicities) that can occur is at most
$\binom{2g-2}{d}.
$ After fixing
$\phi=f_{1}^{*}\phi_{1}
$ and
$f_{1}^{-1}(p_{01}),
$ we can show that there are at most
$\binom{4g-4}{d}
$ possible isomorphic classes of holomorphic maps of degree \(d\) as follows.

Let
$f_{i}:X\to Y_{i}
$ be holomorphic maps \((i=1,2)\).
Suppose that there exist holomorphic differentials
$\phi_{1},\phi_{2}
$ on $Y_{1}$ and $Y_{2}$, respectively, such that
$f_{1}^{*}\phi_{1}=f_{2}^{*}\phi_{2},
$ and suppose that there exists a zero
$p_{01}\ (\text{resp. }p_{02})
$ of
$\phi_{1}\ (\text{resp. }\phi_{2})
$ satisfying
$f_{1}^{-1}(p_{01})=f_{2}^{-1}(p_{02}).
$ We put
$\phi=f_{1}^{*}\phi_{1}=f_{2}^{*}\phi_{2}.
$ Let
$\tilde p_{0}\in f_{1}^{-1}(p_{01})=f_{2}^{-1}(p_{02}).
$
Take a sufficiently small neighbourhood
$U_{\tilde p_{0}}
$ (resp.\ \(U_{p_{0i}}\))
of
$\tilde p_{0}
$ (resp.\ \(p_{0i}\))
so that there is no zero of \(\phi\)
(resp.\ \(\phi_{i}\))
on \(U_{\tilde p_{0}}\)
(resp.\ \(U_{p_{0i}}\))
except
$\tilde p_{0}
$ (resp.\ \(p_{0i}\)),
and that
$f_{i}(U_{\tilde p_{0}})\subset U_{p_{0i}}
\quad (i=1,2).
$ We may choose a local coordinate
$z
$ (resp.\ \(z_{i}\))
on
$U_{\tilde p_{0}}
$ (resp.\ \(U_{p_{0i}}\))
such that
$z(\tilde p_{0})=0
$ (resp.\ \(z_{i}(p_{0i})=0\))
and the differential is written as
$\phi=z^{m}dz
$ (resp.\
$\phi_{i}=z_{i}^{n_{i}}dz_{i}).
$
Since
$f_{1}^{-1}(p_{01})=f_{2}^{-1}(p_{02}),
$ we have
$n_{1}=n_{2}
$ and we denote this by \(n\) for brevity.

We consider the real lines
$\gamma_{i}:[0,a)\to U_{p_{0i}}
$ defined in the coordinates $z_i$ by
$$\gamma_{i}(t)=t\in\mathbb R\quad (i=1,2).$$

For an arbitrary
$\tilde p\in U_{\tilde p_{0}}\setminus\{\tilde p_{0}\}$,
we have

$$\int_{0}^{\tilde p} z^{m}dz
=
\int_{0}^{f_{1}(\tilde p)} z_{1}^{n}dz_{1}
=
\int_{0}^{f_{2}(\tilde p)} z_{2}^{n}dz_{2}.
$$
Hence the number of possible configurations of the lifts of
$\gamma_{1}
$ 
(and therefore also of \(\gamma_{2}\))
in
$U_{\tilde p_{0}}
$ is at most
$m+1.
$ 
Consequently, the total number of possible configurations of all lifts of
$\gamma_{1}
$ is at most
$$\binom{4g-4}{d}.
$$

Let
$\{\tilde p_{0j}\}_{j=1}^{N}
=
f_{1}^{-1}(p_{01})
=
f_{2}^{-1}(p_{02}).
$ Suppose that, for every
$\tilde p_{0j}\in f_{1}^{-1}(p_{01}),
$
 $U_{\tilde p_{0j}}\cap f_{1}^{-1}(\gamma_{1})
=
U_{\tilde p_{0j}}\cap f_{2}^{-1}(\gamma_{2}),
$ that is, the lifts of \(\gamma_{1}\) and \(\gamma_{2}\) coincide locally around each  
$\tilde p_{0j}$.

Then one easily obtains a local conformal map
$h:f_{1}(U_{\tilde p_{0j}})\to f_{2}(U_{\tilde p_{0j}})
$ satisfying
$h\circ f_{1}|_{\cup_{j}U_{\tilde p_{0j}}}
=
f_{2}|_{\cup_{j}U_{\tilde p_{0j}}}.
$
We now show that $h$ extends to a global conformal map from
\(Y_{1}\) to \(Y_{2}\).

For an arbitrary point
$p\in Y_{1},
$
we draw a curve
$c
$ from
$p_{01}
$ to \(p\) avoiding all branch points of \(f_{1}\) except possibly at
$p_{01}
$ and \(p\).
Let
$\tilde c,\tilde c'
$ be two lifts of \(c\) by \(f_{1}\).

Since $h\circ f_1$ is already well defined near each  
$\tilde p_{0j}
\quad (j=1,\dots,N)$,
it follows that
$$f_{2}(\tilde c)=f_{2}(\tilde c').
$$

Therefore  \(h\) is globally well defined on \(Y_{1}\).
It is easy to see that \(h\) is invertible.
\end{proof}

\begin{lemma}\label{count}
We have
$$\sum_{k=2}^d \binom{2g-2}{k}\binom{4g-4}{k}\le
 \binom{6g-6}{2g-2}
 \le
\left( \frac{27}{4} \right)^{2g-2}$$
for any degree $d$.
\end{lemma}
\begin{proof}

By extending the range of summation and applying a classical combinatorial identity, we can bound the sum from above as follows:
\begin{align}
\sum_{k=2}^d \binom{2g-2}{k} \binom{4g-4}{k} 
&< \sum_{k=0}^{2g-2} \binom{2g-2}{k} \binom{4g-4}{4g-4-k} \label{eq:extension} \\
&= \binom{6g-6}{4g-4} \label{eq:vandermonde} \\
&= \binom{6g-6}{2g-2}. \label{eq:symmetry}
\end{align}

Here, the strict inequality \eqref{eq:extension} follows immediately from the fact that all terms in the summation are non-negative, and the upper index satisfies $d < 2g-2$ under the given assumption $g > d > 1$. 
Thus, extending the summation to $k = 2g-2$ and including the positive terms for $k=0,1$ strictly increases the value. 

The subsequent equality \eqref{eq:vandermonde} is a direct application of Vandermonde's convolution identity, which counts the number of ways to choose $4g-4$ elements from a disjoint union of two sets with sizes $2g-2$ and $4g-4$, respectively. 
Finally, the last relation \eqref{eq:symmetry} is obtained via the standard symmetry property of binomial coefficients, $\binom{N}{R} = \binom{N}{N-R}$ with $N = 6g-6$ and $R = 4g-4$.

Next, we show that the binomial coefficient $\binom{6g-6}{2g-2}$ is less than 
or equal to $\left( \frac{27}{4} \right)^{2g-2}$. 
For simplicity, let $n = 2g-2$. 
$$
\binom{3n}{n}
=
\frac{(3n)!}{n!(2n)!}
\le
\frac{(3n)^{3n}}{n^n(2n)^{2n}}
=
\left(\frac{27}{4}\right)^n.
$$

Substituting $n = 2g-2$ into the inequality, we conclude that
\begin{equation*}
\binom{6g-6}{2g-2} \le \left( \frac{27}{4} \right)^{2g-2}.
\end{equation*}
\end{proof}

\section{ A BOUND FOR $\mathcal{I}(X)$}

 Now, we will recall the notion of successive minima which is a basic tool in the geometry of numbers (see e.g. \cite{Ca}).  In $n$-dimensional real vector space, let $\Lambda$ be a lattice, that is, the set of all points 
\begin{equation*}
x=u_{1}a_{1}+\cdot\cdot\cdot+u_{n}a_{n}\in\Lambda \text{\quad }
\end{equation*}
with integers $u_{1},...,u_{n}$, and fixed linearly independent vectors $a_{1},...,a_{n}$. Let $F(x)$ be a distance function, namely, $F(x)$ is a non-negative and continuous function with $F(tx)=tF(x) (t\ge0)$. The $k$-th successive minimum $\lambda_{k}$ of the distance function $F$ with respect to the lattice is the lower bound of the numbers $\lambda$ such that $\{||x||<\lambda\}$ contains $k$ linearly independent lattice points.  In this paper, we take $\Lambda=\{x\in\mathbb{Z}^{2g}\}$, and the distance function $F(x)$ is $||x||=\sqrt{(x,x)_{X}}.$ 

\begin{notation}
Let $a_{1},...,a_{2g}$ be linearly independent points of the lattice such that $||a_{k}||=\lambda_{k}$ for $1\le k\le2g,$ where $\lambda_{k}$ is the $k$-th successive minimum. 
\end{notation}

We now estimate the number of possibilities for the vectors
${}^{t}\mathcal F_i a_k$.
The key observation is that, under the vanishing conditions
${}^{t}\mathcal F_i a_1=\cdots=
{}^{t}\mathcal F_i a_{k-1}=0$,
the vector
${}^{t}\mathcal F_i a_k$
lies in a subspace of codimension $k-1$.
Combining this observation with a standard packing argument, we obtain the following estimate.
\begin{proposition}\label{prop}
Let $f_{i}:X\rightarrow Y_{i}$ be a nonconstant holomorphic map of degree $d_i\leq d$, 
and let $\mathcal{F}_{i}$ be the rational representation of the endomorphism associated with $f_{i}$.  Suppose that, for some $k<2g$,
\begin{equation*}
{}^{t}\mathcal{F}_{i}a_{1}=\cdot\cdot\cdot=^{t}\mathcal{F}_{i}a_{k-1}=0, 
\end{equation*}
and that  $^{t}\mathcal{F}_{i}a_{k}\ne 0$ hold.  Then the number of possible $^{t}\mathcal{F}_{i}a_{k}$
is at most 
$$
(2d+1)^{2g-k+1}.
$$
\end{proposition}

\begin{proof}
In Euclidean space of $n$ dimensions, the volume $V_n(R)$ of the ball of radius $R$ is given by
$$
V_n(R)=\frac{\pi^{\frac{n}{2}}}{\Gamma (\frac{n}{2}+1)}R^n .
$$

Suppose that $^{t}\mathcal{F}_{1}a_{1}\ne 0$ and $^{t}\mathcal{F}_{2}a_{1}\ne 0$.
By Lemma {\ref{lem2}}, 
$$||^{t}\mathcal{F}_i (a_1 )||\le d_i ||a_1 ||=d_i \lambda_1 \quad (i=1,2).
$$
Since $\lambda_1$ is the minimum norm in the lattice, the difference  satisfies
$$||^{t}\mathcal{F}_1 (a_1 )-^{t}\mathcal{F}_2 (a_1 )||\ge \lambda_1
$$ if it is not $0$.
Thus we see that the number of possible $^{t}\mathcal{F}_{i}a_{1}$
is at most 
$$
\frac{V_{2g} \left( d+\frac1{2} \right) }{V_{2g} \left( \frac1{2} \right)}
=(2d+1)^{2g}.
$$

Next, suppose that
$$
{}^{t}\mathcal{F}_{i}a_{1}=0
\qquad\text{and}\qquad
{}^{t}\mathcal{F}_{i}a_{2}\neq 0.
$$
Then
$
{}^{t}\mathcal{F}_{i}a_{2}
$
is orthogonal to \(a_{1}\). Indeed, by Lemma~\ref{lem1},
$$
d\,({}^{t}\mathcal{F}_{i}a_{2},a_{1})_X
=
({}^{t}\mathcal{F}_{i}a_{2},{}^{t}\mathcal{F}_{i}a_{1})_X
=
0.
$$
Hence
$
{}^{t}\mathcal{F}_{i}a_{2}
$
lies in the \((2g-1)\)-dimensional subspace orthogonal to $a_{1}$.
Since \(\lambda_{2}\) is the first successive minimum of the induced
sublattice in this subspace, 
the number of possible $^{t}\mathcal{F}_{i}a_{2}$
is at most 
$$
\frac{V_{2g-1} \left( d+\frac1{2} \right) }{V_{2g-1} \left( \frac1{2} \right)}
=(2d+1)^{2g-1}.
$$

Proceeding by induction, we obtain the conclusion.
\end{proof}

 \textbf{Proof of Theorem \ref{thm}.} 
 We will say that $\mathcal{F}_{i}\in M(2g,2g;\mathbb{Z})$ is of the $k$-th type if $^{t}\mathcal{F}_{1}a_{1}=\cdot\cdot\cdot=^{t}\mathcal{F}_{1}a_{k-1}= 0$ and $^{t}\mathcal{F}_{1}a_{k}\ne0$.

 If we fix the genus $\gamma$ of target surfaces, then the degree $d$ of maps are $\le \frac{g-1}{\gamma-1}$.  Then we note that the number of possible types for associated endomorphisms $\mathcal{F}_{i}$ is at most $2g-2\gamma+1$ since the rank of each $\mathcal{F}_{i}$ is $2\gamma$.  
 Considering all $\gamma>1$, we see $d\leq g-1$.
 Let $\mathcal{F}_{i}$ be of the $k$-th type.  
 By Proposition \ref{prop},
  the number of possible $^{t}\mathcal{F}_{i}a_{k}$
is at most 
$$
(2d+1)^{2g-k+1}\leq (2g-1)^{2g-k+1}.
$$
  
 If ${}^t\mathcal{F}_{1}a_{k}={}^t\mathcal{F}_{2}a_{k}$, then there exist $x_{i}\in H^{1}(Y_{i},\mathbb{Z}) \,(i=1,2)$ such that $^{t}{\frak f}_{1}( \widehat{\Pi}_{1}x_{1}) = {^{t}{\frak f}_{2}( \widehat{\Pi}_{2}x_{2})},$ where $ \widehat{\Pi}_{1}$ (resp. $ \widehat{\Pi}_{2}$) is the period matrix for $ \widehat{J(Y_{1})}$ (resp. $ \widehat{J(Y_{2})})$. Applying Lemma \ref{lem}, 
  we see the number of all isomorphic classes of nonconstant holomorphic maps  satisfies, 
\begin{equation*}
\#\mathcal{I}(X)<\left\{ \left(2g-1\right)^{2g}
+\left(2g-1\right)^{2g-1}+\cdots +\left(2g-1\right)^{4}\right\} \sum_{k=2}^{g-1} \binom{2g-2}{k} \binom{4g-4}{k}. \end{equation*}
Applying Lemma \ref{count} to the right hand side, we obtain
\begin{eqnarray*}
\#\mathcal{I}(X)&\le& 
\left\{ \left(2g-1\right)^{2g}
+\left(2g-1\right)^{2g-1}+\cdots +\left(2g-1\right)^{4}\right\}  \binom{6g-6}{2g-2}\\
&\le&
\left\{ \left(2g-1\right)^{2g}
+\left(2g-1\right)^{2g-1}+\cdots +\left(2g-1\right)^{4}\right\}  \left( \frac{27}{4} \right)^{2g-2}.
\end{eqnarray*}
\qed

\vspace{0.4cm}
 \footnotesize{DEPARTMENT OF MATHEMATICS, SCIENCE TOKYO,
  OHOKAYAMA, MEGURO, TOKYO, 152-8551. JAPAN}  \\
\footnotesize{E-mail address: tanabe.m.f489@m.isct.ac.jp} 

\end{document}